\documentclass[letterpaper, 10pt, conference]{IEEEtran}      % Use this line for a4 paper

\IEEEoverridecommandlockouts                              % This command is only needed if 
\usepackage{hyperref}
\usepackage{graphics} % for pdf, bitmapped graphics files
\usepackage{epsfig} % for postscript graphics files
\usepackage{times} % assumes new font selection scheme installed
\usepackage{amsmath} % assumes amsmath package installed
\usepackage{amssymb}  % assumes amsmath package installed
\usepackage{amsmath} % assumes amsmath package installed
\usepackage{amssymb}  % assumes amsmath package installed
\usepackage{cite}
\usepackage{fullpage}
\usepackage{fancyhdr}
\usepackage{color}

\newtheorem{thm}{Theorem}

\newtheorem{defn}{Definition}

\newtheorem{exmp}{Example}

\newtheorem{prop}{Proposition}

\newcommand{\R}{{\mathbb R}}
\newcommand{\Rnn}{{\mathbb R}_{\ge 0}}
\newcommand{\Rp}{{\mathbb R}_{> 0}}

\newcommand{\tC}{{\mathrm C}}

\newcommand{\cD}{{\mathcal D}}

\newcommand{\cF}{{\mathcal F}}
\newcommand{\cG}{{\mathcal G}}

\newcommand{\cM}{{\mathcal M}}
\newcommand{\cN}{{\mathcal N}}

\newcommand{\cT}{{\mathcal T}}
\newcommand{\cU}{{\mathcal U}}

\newcommand{\bfone}{\mathbf 1}

\newcommand{\diag}{\text{diag}}

\newcommand{\Exp}{\mathbb{E}}
\newcommand{\led}{\leqq_{\rm d}}

\newcommand{\gepsd}{\geqq_{\rm psd}}
\newcommand{\lepsd}{\leqq_{\rm psd}}
\newcommand{\leicx}{\leqq_{\rm icx}}
\newcommand{\geicx}{\geqq_{\rm icx}}
\newcommand{\leidcx}{\leqq_{\rm idcx}}
\newcommand{\cFicx}{{\mathcal F}_{\rm icx}}

\newcommand{\cFidcx}{{\mathcal F}_{\rm idcx}}
\newcommand{\cFicv}{{\mathcal F}_{\rm icv}}

\newcommand{\cFd}{{\mathcal F}_{\rm d}}

\renewcommand{\Pr}{\mathbb P}

\title{On Monotonicity and Propagation of Order Properties}
\author{Aivar Sootla \thanks{The author is an F.R.S.--FNRS fellow with the Institut Montefiore, University of Li\`{e}ge,  B-4000, Li\`{e}ge Belgium {\tt\small aivar.a.sootla@ieee.org}. This work was mostly performed, while the author was a post-doctoral research associate at Imperial College London. The author gratefully acknowledges the support by the EPSRC Science and Innovation Award EP/G036004/1 and is additionally supported by a Belgian National Fund for Scientific Research F.R.S.--FNRS. The author would like to thank Mr Viktor Wolf and Prof Ludger R\"uschendorf for immediate replies to questions related to monotone diffusions and Dr David Angeli for attracting the author's attention to quasimonotonicity.}}
\begin{document}
\maketitle
\begin{abstract}
In this paper, a link between monotonicity of deterministic dynamical systems and propagation of order by Markov processes is established. The order propagation has received considerable attention in the literature, however, this notion is still not fully understood. The main contribution of this paper is a study of the order propagation in the deterministic setting, which potentially can provide new techniques for analysis in the stochastic one. We take a close look at the propagation of the so-called increasing and increasing convex orders. Infinitesimal characterisations of these orders are derived, which resemble the well-known Kamke conditions for monotonicity. It is shown that increasing order is equivalent to the standard monotonicity, while the class of systems propagating the increasing convex order is equivalent to the class of monotone systems with convex vector fields. The paper is concluded by deriving a novel result on order propagating diffusion processes and an application of this result to biological processes.
\end{abstract}
\section{Introduction}
Deterministic monotone systems (cf.~\cite{hirsch2005monotone}) have a considerable number of applications, such as control engineering~\cite{angeli2003monotone} and biology~\cite{sontag2007monotone} to name a couple. Properties of monotone control systems include: easy to compute bounds on reachability sets~\cite{ramdani2010computing}; computation of robust open-loop controls for particular applications~\cite{sootla2015pulsesarxiv}; easy to compute feedback controllers~\cite{meyer2013controllability, chisci2006asymptotic}; availability of structured model reduction methods~\cite{sootla2014projectionI, sootla2014projectionII}.

Monotonicity in the context of Markov processes has been introduced in~\cite{kalmykov1962partial,daley1968stochastically}. Since then this concept has been extensively studied and a number of applications has been discovered. For example, stochastically monotone processes play a central role in the perfect simulation algorithms~\cite{stoyan2002comparison}. Monotone processes are also extensively used in risk theory (cf.~\cite{stoyan2002comparison}), queueing theory (cf.~\cite{stoyan1983comparison}), and financial mathematics (cf.~\cite{bergenthum2007comparison}). Following the nomenclature in~\cite{bergenthum2007comparison}, the term \emph{propagation of order} will be preferred to the term \emph{stochastic monotonicity} in order to avoid confusion with deterministic definitions.  

The so-called $\cFd$, $\cFicx$ and $\cFicv$-orders have received a considerable attention in the literature. At the same time the $\cFicx$ and $\cFicv$-orders are not entirely understood in the context of stochastic processes. Therefore in order to advance the state-of-the-art in the stochastic setting, we take a closer look at the $\cFd$, $\cFicx$ and $\cFicv$ orders in the deterministic one. First, we show that in the deterministic setting the $\cFd$-order propagation is monotonicity in the sense of~\cite{angeli2003monotone}. Then it is shown that the systems propagating the $\cFicx$-order (respectively, the $\cFicv$-order) are monotone systems with a convex (respectively, concave) vector field. Note that, the class of systems propagating the $\cFicx$-order has been implicitly studied in~\cite{rantzer2014conmon} in the context of deterministic control systems, and~\cite{keller2010convexity} in the context of ordinary differential equations. 
However, the proofs presented in this paper are different from the existing ones, since they are influenced by the definitions of order propagation. Using techniques similar to ours it is possible to obtain additional propagation results for different orders as discussed in Section~\ref{s:prop-det}.
%the full version of the paper, which is available online~\cite{Sootla2015POarxiv}.

While studying processes propagating orders, the main application in mind was biological systems. It is well-known that cellular dynamics can be described by a Markov process with Poisson jumps~\cite{van2007stochastic}. Analysis of such processes is complicated due to the dependence of jumps on a Poisson distribution. Hence often Gaussian approximations of Poisson distributions are considered, which can result in Chemical Langevin Equation (CLE)~\cite{gillespie2000langevin} or Linear Noise Approximation (LNA)~\cite{van2007stochastic}. In this paper it is shown that only a collection of decoupled birth-date processes described by an LNA propagate the $\cFd$-order, which limits the scope of possible applications. On the other hand, it is also shown that all unimolecular reactions propagate the $\cFicx$-order and the order is propagated through mean and covariance matrix of the process. Hence the $\cFicx$-order is arguably better suited than the $\cFd$-order for comparison of unimolecular reactions processes. This result constitutes a step forward towards understanding Markov processes propagating orders. 

{\it Notation.} Let $\|\cdot\|_2$ stand for the Euclidean norm in $\R^n$, $X^{\ast}$ stand for a topological dual to $X$. Let $x\succeq_x y$ stand for a partial order in $\R^n$ induced by the non-negative orthant $\Rnn^n$. That is the relation $x\succeq_x y$ is true for vectors $x$ and $y$ if and only if $x_i \ge y_i$, for all $i$ (or $x-y \in \Rnn^n$). Let $x\gg_x y$ be true if and only if $x_i > y_i$, for all $i$ (or $x-y\in\Rp^n$).  For a general definition of the partial order we refer the reader to~\cite{smith2008monotone}. The partial order $u \succeq_u v$ on the space of control signals $u(t)$ is defined as an element-wise comparison $u_i(t) \ge v_i(t)$ for all $i$ and $t$. The notation $X \lepsd Y$ means that the matrix $Y-X$ is positive semidefinite. $\bfone$ stands for the vector of ones. %The operator $\partial_k f(x)$ stands for $\frac{\partial f(x)}{\partial x_k}$, while 
The operator $\nabla f$ stands for the gradient of $f$, while $\nabla^2 f$ stands for the Hessian of $f$. Let $\tC^\infty(\cD\rightarrow\cM)$ be space of the smooth functions acting from $\cD$ to $\cM$. The operator $\Exp[X]$ stands for the expectation of $X$, while $\Exp[X|Y]$ stands for the conditional expectation of $X$ with respect to $\sigma$-algebra generated by $Y$. We write $X\sim Y$ if $X$ and $Y$ have the same probability distribution $\Pr$.
\section{Preliminaries}
\subsubsection{Monotone Control Systems} 
Consider a system
\begin{equation}
  \label{sys:det}
  \dot{x}=f(x,u)
\end{equation}
where $f:\cD\times \cU\rightarrow \R^n$, $\cD\subset\R^n$, $\cU\subset\R^m$, $u$ belongs to the space of $\R^m$-valued measurable functions $\cU_{\infty}$. The associated flow map is $\phi_f: \R_{\ge 0} \times \cD\times \cU_{\infty} \rightarrow \R^n$, which is denoted as $\phi_f(t; x, u)$ and is a solution to the system~\eqref{sys:det} with the initial state equal to $x$, and the control input $u\in\cU_{\infty}$. In order to guarantee the existence and uniqueness of solutions to~\eqref{sys:det}, throughout the paper we will assume that the vector field $f(x,u)$ is continuous in $(x,u)$ and locally Lipschitz continuous in $x$ uniformly on $u$. This means that for each compact sets $C_1\subset \cD$ and $C_2 \subset \cU$, let there exist a constant $k$ such that $\|f(\xi, u) - f(\zeta, u)\|_2 \le k \|\xi - \zeta\|_2$ for all $\xi,\zeta \in C_1$ and $u\in C_2$. 
\begin{defn}\label{def:mon}
  The system is said to be monotone with respect to $\Rnn^n\times\Rnn^m$ if one of these equivalent statements hold (the equivalence is shown in~\cite{angeli2003monotone}): 
\begin{enumerate}
\item\label{con:mon} {\bf Monotonicity.} For all $x\preceq_x y, u\preceq_u v\Rightarrow \phi_f(t;x, u)\preceq_x \phi_f(t;y, v)$ for all $t\in \R_{\ge 0}$. 
\item\label{con:kamke} {\bf Kamke Conditions.} Let $u$, $v$ in $\cU$ such that $u\preceq_u v$.  If $x$, $y$ in $\cD$ are such that $x_i = y_i$ for some $i$ and $x_j \le y_j$ for all $j\ne i$, then $f_i(x,u) \le f_i(y,v)$,
\item\label{con:qm} {\bf Quasimonotonicity.} Let $K = \Rnn^n$, and let $K^{\ast}$ be the topological dual to $K$, that is $K^{\ast}=\left\{ g \in (\R^n)^{\ast}| g(K) \ge 0\right\}$. For $x$, $y\in\cD$, $u$, $v\in\cU$ such that $x\preceq_x y$, $u\preceq_u v$ and $g \in K^{\ast}$ such that $g(x) = g(y)$, we have $g(f(x, u)) \le g(f(y, v))$.  
\end{enumerate}  
\end{defn}
A generalisation can be defined with respect to any orthant by mapping this orthant onto the positive one by a linear transformation $T: \R^n \rightarrow \R^n$, where $T = \diag((-1)^{\varepsilon_1}, \dots, (-1)^{\varepsilon_n})$ for some $\varepsilon_i$. 

\subsubsection{Markov Processes Propagating Orders}
Here we follow the development of stochastic monotonicity theory in~\cite{stoyan2002comparison,shaked2007stochastic}. We consider only uncontrolled processes for simplicity. But before we proceed, we require a few definitions.
\begin{defn}
 Let $g:\R^n \rightarrow \R^m$ be called increasing if for two vectors $x$,  $y$ such that $x\preceq_x y$, we have $g(x) \preceq_x g(y)$. 
 
Let $g:\R^n \rightarrow \R^m$ be called convex with respect to the order $\preceq_x$ if for two vectors $x$,  $y$ and a scalar $\lambda \in[0, 1]$,  we have 
\[
g(\lambda x + (1-\lambda) y) \preceq_x \lambda g(x) + g((1-\lambda) y).
\] 

Let $g:\R^n \rightarrow \R^m$ be called directionally convex with respect to the order $\preceq_x$ if for any $x_1 \preceq_x [ x_2, x_3] \preceq_x x_4$ and $x_1 + x_4 = x_2 + x_3$, we have:
\[ 
g(x_2) + g(x_3) \preceq_x g(x_1) + g(x_4).
\]

The classes $\cFd$, $\cFicx$ and $\cFidcx$ are classes of increasing, increasing convex and increasing directionally convex functions from $\tC^{\infty}(\R^n\rightarrow \R^m)$, respectively. 
\end{defn}

If the partial order is induced by the positive orthant $\R^n$, then the definition of convexity with respect to the order is a generalisation of convexity to vector-valued functions. However, if the orthant is not positive, then some individual $g_i(x,u)$ can be actually concave functions of the arguments, but still convex with respect to the order. For example, the function $f(x, y) = -x^2 + y$ is convex with respect to the order $\succeq_x$, if the relation $x\succeq_x y$ is implied by $x_1 \le x_2$, $y_1 \ge y_2$ (this order is induced by the orthant $\diag([-1, 1])\Rnn^n$). However, $f(x,y)$ is clearly concave in the classical sense. Note finally that a function from $\cFicx$ cannot increase and be convex with respect to two different partial orders.%, since increase and convexity should be with respect to the same order. 

Directionally convexity of a twice differentiable function $g$ in the standard order can be checked by inspecting the sign pattern of the Hessian. That is $g\in\cFidcx$ if and only if $\frac{\partial^2 g_k}{\partial x_i \partial x_j} \ge 0$ for all $k$, $i$, $j$. 

\begin{defn}\label{def:order} We say that $X \led Y$ (resp., $X \leicx Y$) [resp. $X\leidcx Y$], if 
\[
\Exp[g(X)] \le \Exp[g(Y)]
%\int\limits_{E} g(t) P_X(d t) \le \int\limits_{E} g(t) P_Y(d t)
\]
holds for all $g$ in $\cFd$ (resp. $\cFicx$) [resp. $\cFidcx$] for which the integrals exist.
\end{defn}

Besides the classes $\cFd$, $\cFicx$, $\cFidcx$ different classes of functions can be used~\cite{shaked2007stochastic}. For example, the order induced by the class $\cFicv$ of increasing concave functions. We will leave a detailed discussion on $\cFicv$ and $\cFidcx$ orders beyond the scope of this paper. %, however, some remarks will be available in~\cite{Sootla2015POarxiv}.
 The orders $\cFd$, $\cFicx$ have the following probabilistic interpretations and properties. 
\begin{prop}\label{prop:st-o} Let $X$, $Y$ be two random variables.
\begin{enumerate}
  \item $X\led Y$, if and only if $X \le Y$ almost surely. 
  \item $X\leicx Y$ if and only if $ X \led \Exp[Y | X] $
  \item Let $X\sim\cN(m_X, \Sigma_X)$, $Y\sim\cN(m_Y, \Sigma_Y)$. Then  $X\led Y$ if and only if $m_X \preceq_x m_Y$ and $\Sigma_X = \Sigma_Y$. If $m_X \preceq_x m_Y$ and $\Sigma_X \lepsd \Sigma_Y$, then $X\leicx Y$.
\end{enumerate}
  \end{prop}
The  proofs or references to such can be found in~\cite{stoyan2002comparison}, ~\cite{shaked2007stochastic}. 
Instead, we try to provide some intuition behind these results. 

The proof of 1) for the univariate case is perhaps the easiest to reproduce. Consider the function $P(X > t)$, which is equal to $\Exp f_t(X)$, where $f_t(X)$ is an indicator function of the set $X>t$ and hence an increasing function. Therefore if $X\led Y$, then $X \le Y$ almost surely. Note that an indicator function is not smooth, but it is also possible to show the same result for smooth increasing functions. From $X \le Y$ almost surely, it follows directly that for all increasing functions $f(X) \le f(Y)$ almost surely, and hence clearly $X \led Y$. Note that in order to prove the necessity, a {\it nonlinear} increasing function was used, hence the usual definition monotonicity cannot be extended directly to the stochastic case, if one wants to preserve the order in the almost surely sense.
  
The sufficiency of 2) follows directly from Jensen's inequality as follows
\[
\Exp[g(X)]  \le \Exp[g(\Exp[Y | X])] \le  \Exp[\Exp[g(Y) | X]] =  \Exp[g(Y)]
\]
The necessity proof is omitted, since it is much more complicated than sufficiency. Note, however, that the necessity requires that $\Exp[g(Y)] \le \Exp[g(X)]$ for nonlinear convex $g$. 

The $\cFd$-order is the comparison in the almost surely sense, therefore comparing variables with different variances is problematic, which is confirmed in the point 3). On the other hand the $\cFicx$-order is a comparison of conditional expectations. Hence if $X$ generates a weaker $\sigma$-algebra then $Y$, then the variables are comparable. This condition for the Gaussians is translated into the comparison of the variances. This means that the $\cFicx$-order is {\it weaker} than $\cFd$ for Gaussian distributions.

The presented below definition of order propagating Markov processes is slightly different than the one in~\cite{bergenthum2007comparison}. This, however, does not affect any of the derivations for stochastic differential equations and does not affect the purpose of this definition. 
\begin{defn}\label{def:stoch-mon}
Let $\cF$ be one of $\{\cFd$, $\cFicx$, $\cFidcx\}$. A Markov process $X(t, x_0)$ with $X(0) = x_0$ propagates the order induced by the class $\cF$ , if for any $g$ from $\cF$ , the function $\Exp[g(X(t, X_0))| X(0) = x_0]$ also belongs to $\cF$ as a function of $x_0$.
\end{defn}

Propagation of $\cFd$, $\cFicx$ and $\cFidcx$ orders is usually referred to as {\it stochastic monotonicity}, {\it icx-monotonicity}, and {\it idcx-monotonicity}, respectively. We prefer using the term propagation of order in favour of the term stochastic monotonicity in order to avoid the use of the word "stochastic" in the context of deterministic systems. 

\section{Monotonicity and Order Propagation in the Deterministic Setting \label{s:prop-det}}

Deterministic systems can be seen as Markov processes with probability densities concentrated at one point. Hence, it is straightforward to formulate the order propagation in the deterministic setting. 
\begin{defn}
Let $\cF$ be one of $\{\cFd$, $\cFicx$, $\cFidcx\}$. The system~\eqref{sys:det} propagates the $\cF$-order if for any $g$ from  $\cF$, $x\in\cD$ and $u\in\cU_{\infty}$, we have that $g(\phi_f(t; x, u))$ belongs to $\cF$.
\end{defn}

The classes of functions $\cFd$, $\cFicx$, $\cFidcx$ are proper, convex cones in the space $\tC^{\infty}$. This indicates a clear connection between $\cFd$-order propagation and quasimonotonicity. However, in the stochastic case, it is necessary to include nonlinear functions in the classes $\cF$, $\cFicx$ and $\cFidcx$. Hence, we proceed without the linearity assumption on the functions $g$. Firstly, we establish the equivalence of $\cFd$-order propagation and monotonicity. 
\begin{thm} \label{thm:d-po}Consider the system~\eqref{sys:det}, then the following statements are equivalent:
\begin{enumerate}
\item \label{con:kamke-po} For $x$, $y\in\cD$, $u$, $v\in\cU$ such that $x \preceq_x y$, $u \preceq_u v$ and  $g\in\cFd$ such that $g(x) = g(y)$, we have $(\nabla g(x))^T f(x, u) \le (\nabla g(y))^T f(y, v)$.
%\item \label{con:mon-thm} The system~\eqref{sys:det} is monotone with respect to $\Rp^n\times\Rp^m$.
\item \label{con:po} The system~\eqref{sys:det} propagates the $\cFd$-order. That is for all $g\in\cFd$,  we have that $g(\phi_f(t; \cdot, \cdot))\in\cFd$ for all $t$.
\item \label{con:mon} The system~\eqref{sys:det} is monotone according to Definition~\ref{def:mon}.
\end{enumerate}
\end{thm}
\begin{proof}
%The proof can be found in~\cite{Sootla2015POarxiv}. 
All the implications can be shown in a few lines, by applying monotonicity results, however, we produce direct proofs without involving monotonicity results. This is done to illustrate the tools, which are used in the sequel.

\ref{con:kamke-po}) $\Rightarrow$~\ref{con:po}). Let $u\preceq_u v$,  $x\preceq_x y$. Consider a system $f_m(x,u) = f(x, u) + \bfone/m$, where $\bfone$ is a vector of ones. Let the initial condition be $y^{m} = y + 1/m \bfone$, and the flow corresponding to the vector field $f_m$ be $y^m(t) = \phi_m(t; y^{m}, v)$. Let also $x(t) = \phi_f(t; x, u)$. First, we show by contradiction that $g(x(t)) < g(y^{m}(t))$ holds for all functions $g$ from $\cFd$.  For a small $t$ the condition $g(x(t)) < g(y^m(t))$ obviously holds due to continuity of solutions to~\eqref{sys:det}. Assume there exists a time $\tau$ and a nontrivial function $\eta\in\cFd$ such that $\eta(x(s)) < \eta(y^m(s))$ for all $0\le s < \tau$ and $\eta(x(\tau)) = \eta(y^m(\tau))$ for some $m$. This implies that 
\[
\frac{d}{dt}\eta(x(t))\Bigl|_{t=\tau} \ge \frac{d}{dt}\eta(y^m(t))\Bigl|_{t=\tau}.
\]

On the other hand
{
\begin{multline*}
\frac{d}{dt}\left(\eta(y^m(t))-\eta(x(t)) \right)\Bigl|_{t=\tau} = \\
(\nabla\eta(y^m(\tau))^T (f(y^m(\tau),v) + \bfone/m) - \\
(\nabla\eta(x(\tau)))^T f(x(\tau),u)\ge^{(\ast)} \\ (\nabla\eta(y^m(\tau))^T \bfone/m >^{(\ast\ast)}0,
\end{multline*}}
where the inequality $(\ast)$ is due to~\eqref{con:kamke-po} and the inequality  $(\ast\ast)$ holds since $\eta$ is a nontrivial increasing function. Hence we have
\[
\frac{d}{dt}\left(\eta(y^m(t))-\eta(x(t)) \right)\Bigl|_{t=\tau} >0,
\]
and arrive at the contradiction. This implies that $g(\phi_f(t; x, u)) < g(\phi_m(t; y^m, v))$ for all $g\in\cFd$ and all finite $m$. Finally, by continuity of solutions with $m\rightarrow\infty$ we have that $g(\phi_f(t; x, u)) \le g(\phi_f(t; y, v))$, if $u\preceq_u v$,  $x\preceq_x y$,  $g\in\cFd$.

\ref{con:po}) $\Rightarrow$~\ref{con:kamke-po}). Let $g\in\cFd$  be such that for $x$, $y\in\cD$, such that $x \preceq_x y$, we have $g(x) = g(y)$. Let $u$, $v\in\cU$ $u \preceq_u v$. Due to~\ref{con:po}), we have that 
\[ 
g(\phi_f(t; x, u)) \le g(\phi_f(t; y, v))~\forall t\ge 0.
\]
Consequently
\begin{gather*}
\frac{d}{dt}(g(\phi_f(t; y, v)) -g(\phi_f(t; x, u)))\Bigl|_{t = 0} = \\
(\nabla g(y) )^T f(y, v) - (\nabla g(x))^T f(x, u) \ge 0,
%(\nabla g(y))^T (f(y, v) - f(x, u)) \ge 0
\end{gather*}
which is the condition~\ref{con:kamke-po}).

\ref{con:kamke-po})$\Rightarrow$\ref{con:mon}) The claim follows by applying the condition~\ref{con:po}) with $g_i(x) = x_i$.

\ref{con:mon})$\Rightarrow$\ref{con:po}) Monotonicity implies that $\phi_f(t;x,u)$ is an increasing function in $(x,u)$. Since a composition of increasing functions is  increasing,  $g(\phi_f(t;x,u))$ is in $\cFd$.
\end{proof}

The next theorem provides necessary and sufficient conditions for the system to propagate the $\cFicx$-order.
\begin{thm} \label{thm:icx:po} Consider the system~\eqref{sys:det}. The following statements are equivalent:
\begin{enumerate}
\item \label{con:icx-kamke-po} 
\begin{enumerate}
\item \label{con:icx-kamke-inc} Let $x$, $y\in\cD$, $u$, $v\in\cU$, $g\in\cFicx$ be such that $x\preceq_x y$, $u\preceq_u v$. If $g(x) = g(y)$,  then
\[
(\nabla g(x))^T f(x, u) \le (\nabla g(y))^T f(y, v)
\]
\item \label{con:icx-kamke-cvx} Let $x$, $y$ $z\in\cD$,  $g\in\cFicx$ be such that  $g(z) = \mu g(x) + (1-\mu) g(y)$ for some $\mu\in[0,~1]$, then for all $u$, $v\in\cU$ we have:%$z\preceq_x \lambda x + (1-\lambda) y$,
\begin{multline*}
(\nabla g(z))^T f(z, \mu u + (1-\mu)v) \le \\ 
\mu (\nabla g(x))^T f(x, u) + (1-\mu) (\nabla g(y))^T f(y, v)
\end{multline*}
\end{enumerate}
\item \label{con:icx-po}The system~\eqref{sys:det} propagates the $\cFicx$-order. That is for any $g$ from  $\cFicx$, we have that $g(\phi_f(t; \cdot, \cdot))$ belongs to $\cFicx$ for all $t$.
\item \label{con:icx-kamke-useful} The system~\eqref{sys:det} is monotone according to Definition~\ref{def:mon} and $f_i(x,u)$ are convex functions in $(x,u)$.
\end{enumerate}
\end{thm}
\begin{proof} %The proof can be found in~\cite{Sootla2015POarxiv}. 
Throughout the proof of Theorem~\ref{thm:icx:po}, we will use the following notations $x^\lambda = \lambda x + (1-\lambda) y$, $u^{\lambda} = \lambda u + (1-\lambda)v$ and $z^\lambda(t) =\phi_f(t; x^\lambda, u^{\lambda})$.

\ref{con:icx-kamke-po}) $\Rightarrow$ \ref{con:icx-po}). Due to Theorem~\ref{thm:d-po} the condition~\eqref{con:icx-kamke-inc} implies that the flow $g(\phi_f(t;x,u))$ is an increasing function of $x$ and $u$ for all $g\in\cFicx$. Now we need to show that $g(\phi_f(t;x,u))$ is a convex function of $x$ and $u$.

Let $x^m = x + \bfone/m$, $y^m = y + \bfone/m$. Let also $\phi_m(t, x,u)$ be the flow generated by the system $\dot x = f(x,u) + \bfone/m$, $x^m(t) = \phi_m(t; x^m, u)$, $y^m(t) = \phi_m(t; y^m, v)$. First, we show that the condition
\begin{equation}
g(z^\lambda(t)) < \lambda g(x^m(t))+(1-\lambda)g(y^m(t)), \label{eq:icx-all_g}
\end{equation}
is valid for all $t$, all $\lambda\in[0,~1]$ and all functions $g\in\cFicx$. For small times $t$, we have 
\begin{multline*}
g(z^\lambda(t)) =  g(\phi_f(t; x_\lambda, u_\lambda)) < \\
g(\phi_f(t; x_\lambda + \bfone/m, u_\lambda)) \le \\
\lambda g(\phi_f(t; x^m, u)) + (1-\lambda) g(\phi_f(t; y^m, v)),
\end{multline*}
where the last inequality holds for $t=0$. Hence due to continuity of solutions, the condition~\eqref{eq:icx-all_g} holds for $t=0$ and its vicinity. Assume there exists a time $\tau$ and a nontrivial function $\eta\in\cFicx$ such that for all $\lambda\in[0,~1]$ and for all $0\le s < \tau$
\begin{equation}
\label{eta-cvx}
\eta(z^\lambda(s)) < \lambda \eta(x^m(s))+(1-\lambda)\eta(y^m(s)),
\end{equation}
while at the time $\tau$ for some $\mu\in[0,~1]$ we have:
\[
\eta(z^\mu(t)) = \mu \eta(x^m(\tau))+(1-\mu)\eta(y^m(\tau)). 
\]
This implies that
\begin{multline}\label{cont:ass}
\frac{d}{dt}(\eta(z^{\mu}(t)) - \mu \eta(x^m(\tau)) \\
-(1-\mu)\eta(y^m(\tau))) \Bigl|_{t = \tau}  \ge 0.
\end{multline}

On the other hand we have that 
\begin{multline*}
\frac{d}{dt}(\eta(z^{\mu}(t)) - \mu \eta(x^m(\tau)) 
-(1-\mu)\eta(y^m(\tau))) \Bigl|_{t = \tau}  = \\ 
(\nabla\eta(z^{\mu}(\tau)))^T f(z_\mu(\tau), u_\mu) - \\
\mu (\nabla \eta(x^m(\tau)))^T (f(x^m(\tau), u) + \bfone/m) - \\
(1-\mu) \nabla \eta(y^m(\tau))^T (f(y^m(\tau), v) + \bfone/m) \le^{(\ast)} \\ 
-(\mu \nabla \eta(x^m(\tau)) + (1-\mu) \nabla \eta(y^m(\tau)))^T \bfone/m < 0,
\end{multline*}
where the inequality $(\ast)$ is due to  the condition~\ref{con:icx-kamke-cvx}). We arrive at the contradiction with the condition~\eqref{cont:ass}. Hence the inequality \eqref{eq:icx-all_g} holds for all $t>0$, all $\lambda\in[0,~1]$ and all $g$. By continuity of solutions with $m\rightarrow\infty$ we have that 
\[
g(z^\lambda(t)) \le \lambda g(\phi_f(t; x, u))+(1-\lambda)g(\phi_f(t; y, v)), 
\]
which completes the proof.

\ref{con:icx-po}) $\Rightarrow$ \ref{con:icx-kamke-po}). By Theorem~\ref{thm:d-po} the condition~\ref{con:icx-kamke-inc}) is implied by the fact that the flow $g(\phi_f(t;x,u))$ is an increasing function of $x$ and $u$ for all $g\in\cFicx$. Now we need to show that if $g(\phi_f(t;x,u))$ is a convex function of $x$ and $u$ then \ref{con:icx-kamke-cvx}) is fulfilled.

Let $x$, $y$ $z\in\cD$, $u$, $v\in\cU$, $g\in\cFicx$ be such that $z\preceq_x \lambda x + (1-\lambda) y$, $u\preceq_u v$, and $g(z) = \mu g(x) + (1-\mu) g(y)$ for some $\mu\in [0,~1]$. Due to convexity of the flow the following holds for all $t$:
\[
g(z(t)) \le \mu g(\phi_f(t; x, u))+(1-\mu)g(\phi_f(t; y, v)), 
\]
Take the derivative at $t=0$ and obtain
\begin{multline}
(\nabla g(z))^T f(z, u^\mu) \le \mu (\nabla g(x))^T f(x,u)+\\(1-\mu)(\nabla g(y))^T f(y, v), 
\end{multline}
This implies the condition \ref{con:icx-kamke-cvx}) and completes the proof.

\ref{con:icx-kamke-po}) $\Rightarrow$ \ref{con:icx-kamke-useful}) The claim follows by testing the conditions~\ref{con:icx-kamke-po}) on $g_i(x) = x_i$.

\ref{con:icx-kamke-useful}) $\Rightarrow$  \ref{con:icx-po}) 
By repeating the proof for \ref{con:icx-kamke-po}) $\Rightarrow$ \ref{con:icx-po}) with $\eta_i(x) = x_i$, we get that the flow 
$\phi_f(t; x, u)$ is an increasing convex function in $x$, $u$. Since any composition of two increasing convex functions is known to be increasing convex, the function $g(\phi_f(t; x, u))$ belongs to $\cFicx$ for all $t>0$, which completes the proof. 
\end{proof}

Dynamical and control systems propagating the $\cFicx$ order were implicitly studied in~\cite{keller2010convexity} and~\cite{rantzer2014conmon}, respectively. The major difference between~\cite{keller2010convexity} and~\cite{rantzer2014conmon} and the presented proof is that it employs infinitesimal characterisations~\eqref{con:icx-kamke-po} and common techniques in the monotonicity theory. Furthermore, the presented proof can be modified in a straightforward manner in order to accommodate propagation of different orders. For instance, it is straightforward to modify our proof in order to obtain similar results for the $\cFicv$-order by flipping the inequalities and changing $1/m$ to $-1/m$. Moreover, it is possible to enforce different inequalities on the flow, for example it is possible to enforce directional convexity.% (see~\cite{Sootla2015POarxiv}).
\begin{thm} \label{thm:idcx:po} Consider the system~\eqref{sys:det}. The following statements are equivalent:
\begin{enumerate}
\item \label{con:idcx-kamke-po} 
\begin{enumerate}
\item \label{con:idcx-kamke-inc} Let $x$, $y\in\cD$, $u$, $v\in\cU$, $g\in\cFidcx$ be such that $x\preceq_x y$, $u\preceq_u v$. If $g(x) = g(y)$,  then
\[
(\nabla g(x))^T f(x, u) \le (\nabla g(y))^T f(y, v)
\]
\item \label{con:idcx-kamke-dcvx} Let $x_1$, $x_2$ $x_3$, $x_4\in\cD$, be such that $x_1 \preceq_x x_2, x_3 \preceq_x x_4$ and $g(x_1) + g(x_4) = g(x_2) + g(x_3)$,  $g\in\cFidcx$. Then for all $u_1$, $u_2$, $u_3$, $u_4\in\cU$  such that $u_1 \preceq_u u_2, u_3 \preceq_x u_4$ and $u_1 + u_4 = u_2 + u_3$ we have:%$z\preceq_x \lambda x + (1-\lambda) y$,
\begin{multline*}
\nabla^T g(x_2) f(x_2, u_2) + \nabla^T g(x_3) f(x_3, u_3) \le \\
\nabla^T g(x_1) f(x_1, u_1) + \nabla^T g(x_4) f(x_4, u_4) 
\end{multline*}
\end{enumerate}
\item \label{con:idcx-po}The system~\eqref{sys:det} propagates the $\cFidcx$-order. That is for any $g$ from  $\cFidcx$, we have that $g(\phi_f(t; \cdot, \cdot))$ belongs to $\cFidcx$ for all $t$.
\item \label{con:idcx-kamke-useful} The system~\eqref{sys:det} is monotone according to Definition~\ref{def:mon}, and $f_i(x,u)$ are directionally convex functions in $(x,u)$ for all $i$.
\end{enumerate}
\end{thm}
\begin{proof}
\ref{con:idcx-kamke-po}) $\Rightarrow$ \ref{con:idcx-po}). Due to Theorem~\ref{thm:d-po} the condition~\eqref{con:idcx-kamke-inc} implies that the flow $g(\phi_f(t;x,u))$ is an increasing function of $x$ and $u$ for all $g\in\cFidcx$. Now we need to show that $g(\phi_f(t;x,u))$ is a directionally convex function of $x$ and $u$.

Consider initial conditions $x_1$, $x_2$, $x_3$, $x_4$, such that  $x_1 \preceq_x x_2, x_3\preceq_x x_4$ and $x_1 + x_4 = x_2 + x_3$, as well as control signals  $u_1$, $u_2$, $u_3$, $u_4$, such that  $u_1 \preceq_u u_2, u_3\preceq_x u_4$ and such that $u_1 + u_4 = u_2 + u_3$. Let $x^m = x_4 + 1/m \bfone$, $x_4^m(t) = \phi_m(t, x_4^m, u_4)$ be the flow generated by the system $\dot x = f(x,u)+1/m\bfone$. Let also $x_i(t) = \phi_f(t; x_i, u_i)$. First, we show that the condition
\begin{equation}
g(x_2(t)) + g(x_3(t)) < g(x_1(t)) + g(x_4^m(t))  \label{eq:idcx-all_g}
\end{equation}
is valid for all $t$ and all functions $g\in\cFidcx$. For $t=0$, we have
\begin{equation}
g(x_2(t)) + g(x_3(t)) \le g(x_1(t)) + g(x_4(t))  %\label{eq:idcx-all_g}
\end{equation}
and $g(x_4(t)) < g(x_4^m(t))$. Hence, the inequality~\eqref{eq:idcx-all_g} is satisfied for small $t=0$. Assume there exists a time $\tau$ and a nontrivial function $\eta\in\cFidcx$ such that for all $0\le s < \tau$
%\begin{multline} 
\begin{equation}
\label{eta-dcvx}
\eta(x_2(t)) + \eta(x_3(t)) < \eta(x_1(t)) + \eta(x_4^m(t)) 
\end{equation}
while at some time $\tau$ we have:
%\begin{multline}
\[
\eta(x_2(\tau)) + \eta(x_3(\tau)) = \eta(x_1(\tau)) + \eta(x_4^m(\tau)) 
\]
This implies that
{\small
\begin{multline*}%\label{cont:ass}
\frac{d}{dt}(\eta(x_2(\tau)) + \eta(x_3(\tau)) - \eta(x_1(\tau)) - \eta(x_4^m(\tau))) \Bigl|_{t = \tau}  \ge 0.
\end{multline*}
}

On the other hand we have that 
\begin{multline*}
\frac{d}{dt}(\eta(x_2(\tau)) + \eta(x_3(\tau)) - \eta(x_1(\tau)) - \eta(x_4^m(\tau))) \Bigl|_{t = \tau}  = \\
\nabla^T \eta(x_2(\tau)) f(x_2(\tau), u_2(\tau)) + \\ 
\nabla^T \eta(x_3(\tau)) f(x_3(\tau), u_3(\tau)) - \\
\nabla^T\eta(x_1(\tau)) f(x_1(\tau), u_1(\tau)) - \\
\nabla^T \eta(x_4(\tau)) (f(x_4(\tau), u_4(\tau)) + \bfone/m) \le^{(1)} \\
-\nabla^T \eta(x_4(\tau))\bfone/m < 0
\end{multline*}
where the inequality $(1)$ is due to  the condition~\ref{con:idcx-kamke-dcvx}). We arrive at the contradiction. Hence the inequality \eqref{eq:idcx-all_g} holds for all $t>0$ and all $g$. By continuity of solutions with $m\rightarrow\infty$ we have that 
\[
g(x_2(t)) + g(x_3(t)) \le g(x_1(t)) + g(x_4(t)) 
\]
which completes the proof.

\ref{con:idcx-po}) $\Rightarrow$ \ref{con:idcx-kamke-po}). By Theorem~\ref{thm:d-po} the condition~\ref{con:idcx-kamke-inc}) is implied by the fact that the flow $g(\phi_f(t;x,u))$ is an increasing function of $x$ and $u$ for all $g\in\cFidcx$. Now we need to show that if $g(\phi_f(t;x,u))$ is a directionally convex function of $x$ and $u$ then \ref{con:idcx-kamke-dcvx}) is fulfilled.

Let $x_1$, $x_2$ $x_3$, $x_4\in\cD$, be such that $x_1 \preceq_x x_2, x_3 \preceq_x x_4$ and $x_1 + x_4 = x_2 + x_3$,  $g\in\cFidcx$, let also $u_1$, $u_2$, $u_3$, $u_4\in\cU$  be such that $u_1 \preceq_u u_2, u_3 \preceq_x u_4$ and $u_1 + u_4 = u_2 + u_3$. Due to directional convexity of the flow the following holds for all $t>0$:
\[
g(x_2(t)) + g(x_3(t)) \le g(x_1(t)) + g(x_4(t)), 
\]
Take the derivative at $t=0$ and obtain
\begin{multline}
\nabla g(x_2)^T f(x_2, u_2) + \nabla g(x_3)^T f(x_3, u_3) \le \\
\nabla g(x_1)^T f(x_1, u_1) + \nabla g(x_4)^T f(x_4, u_4) 
\end{multline}
This implies the condition \ref{con:idcx-kamke-dcvx}) and completes the proof.

\ref{con:idcx-kamke-po}) $\Rightarrow$ \ref{con:idcx-kamke-useful}) The claim follows by testing the conditions~\ref{con:idcx-kamke-po}) on $g_i(x) = x_i$.

\ref{con:idcx-kamke-useful}) $\Rightarrow$  \ref{con:idcx-po}) 
By repeating the proof for \ref{con:idcx-kamke-po}) $\Rightarrow$ \ref{con:idcx-po}) with $g_i(x) = x_i$, we get that the flow $\phi_f(t; x, u)$ is an increasing directionally convex function in $x$, $u$. Since any composition of two increasing directionally convex functions is known to be increasing directionally convex (cf.~\cite{shaked2007stochastic}), the function $g(\phi_f(t; x, u))$ belongs to $\cFidcx$ for all $t>0$, which completes the proof. 
\end{proof}
\begin{exmp} Consider a model of a toggle switch, which was a pioneering circuit in synthetic biology~\cite{Gardner00}. The genetic toggle switch is composed of two mutually repressive genes \emph{LacI} and \emph{TetR}. We consider a control-affine model, which is consistent with a toggle switch actuated by light induction~\cite{Levskaya09}. The dynamical equations can be written as follows.
  \begin{equation}
    \label{eq:ts-par}
    \begin{aligned}
      \dot x_1 &= \frac{p_1}{1 + (x_2/p_2)} - p_3 x_1 + u, \\
      \dot x_2 &= \frac{p_4}{1 + (x_1/p_5)} - p_{6} x_2, 
    \end{aligned}
  \end{equation}
It is easy to verify that for nonnegative $p_i$ the model is monotone with respect to orthant $\diag([1,~-1])\Rnn^2$ and the nonlinearties are convex. However, this model is not propagating the $\cFicx$-order, since the vector is not convex with respect to the orthant $\diag([1,~-1])\Rnn^2$! Indeed, using the change of variables $y_1 = x_1$, $y_2 = - x_2$, it is straightforward to obtain the following cooperative model in the standard order:
  \begin{equation}
    \label{eq:ts-par-2}
    \begin{aligned}
      \dot y_1 &= f_1(y_1, y_2) = \frac{p_1}{1 - (y_2/p_2)} - p_3 y_1 + u, \\
      \dot y_2 &= f_2(y_1, y_2) = \frac{- p_4}{1 + (y_1/p_5)} - p_{6} y_2, 
    \end{aligned}
  \end{equation}
where $y_1(t)\ge 0$, $y_2(t) \le 0$ for all $t\le 0$. It is easy to see that the function $f_1$ is convex, while the function $f_2$ is concave in the orthant $\diag([1,~-1])\Rnn^2$. Therefore Theorem~\ref{thm:icx:po} cannot be applied to the system~\eqref{eq:ts-par-2}.

Consider another change of variables $z_1 = x_1$, $z_2 = 1/x_2$, and obtain the following dynamical system:
  \begin{equation}
    \label{eq:ts-par-3}
    \begin{aligned}
      \dot z_1 &= g_1(z_1, z_2) = \frac{p_1 z_2}{1/p_2 + z_2} - p_3 z_1 + u, \\
      \dot z_2 &= g_2(z_1, z_2) = \frac{- z_2^2 p_4}{1 + z_1/p_5} + p_{6} z_2, 
    \end{aligned}
  \end{equation}
where 
\begin{gather*}
\frac{\partial g_1(z_1, z_2)}{\partial z_2} =  \frac{p_1/p_2}{ (1/p_2 + z_2)^2} \\ 
\frac{\partial g_2(z_1, z_2)}{\partial z_1} = \frac{z_2^2 p_4/p_5}{(1 + z_1/p_5)^2}  \\
\nabla^2 g_1(z_1, z_2) = \begin{pmatrix}
0 & 0 \\ 0 & \dfrac{- 2 p_1/p_2}{ (1/p_2 + z_2)^3}
\end{pmatrix}   \\
\nabla^2 g_2(z_1, z_2) = \begin{pmatrix}
\dfrac{- 2 z_2^2 p_4/p_5^2}{(1 + z_1/p_5)^3} & \dfrac{2 z_2 p_4/p_5}{(1 + z_1/p_5)^2} \\ \dfrac{2 z_2 p_4/p_5}{(1 + z_1/p_5)^2} & \dfrac{- 2 p_4}{1 + z_1/p_5}
\end{pmatrix}   
\end{gather*}
It is straightforward to verify that the vector field satisfies the monotonicity condition, while the Hessians have one zero eigenvalue and one nonpositive eigenvalue. This means that the vector field is concave in the standard order. This implies that the original system~\eqref{eq:ts-par} has a very specific property: every component of the flow mapping is convex in $x_1$ and concave in $x_2$ in the order induced by the orthant $\diag([1,~-1])\Rnn^n$.
\end{exmp}

\section{Remarks on Diffusion Processes Propagating the $\cFd$ and $\cFicx$ orders}
Consider the diffusion process in the form of the stochastic differential equation
\begin{multline}
X(t,X_0) = X_0 + \int_0^t f(X(\tau, X_0)) d\tau + \\
\int_0^t \sigma(X(\tau, X_0)) d W, \label{eq:sde}
\end{multline}
where $W$ is Brownian motion and the integrals are of the It\^{o} type. Let also $c(x)$ be equal to $\sigma(x)^T \sigma(x)$.  Throughout the section, it is assumed that solutions to~\eqref{eq:sde} exist, which is usually implied by some additional assumptions on $f$ and $c$ (see~\cite{oksendal2003stochastic}).
\begin{prop}[Theorem~5.3. in~\cite{herbst1991diffusion}] \label{prop:fd-prop}
Let the solutions to~\eqref{eq:sde} exist. The process $X(t,X_0)$ is propagating the $\cFd$-order if and only if the functions $f_i(x)$ are increasing in $x_k$ for all $k\ne i$, the function $c_{i j}(x)$ depends only on $x_i$ and $x_j$.
\end{prop}

Regarding the propagation of the $\cFicx$-order the literature is not as rich as in the $\cFd$-order case. In fact, only sufficient conditions on propagation of the $\cFicx$-order by a stochastic diffusion are known~\cite{bergenthum2007comparison}. The next result is a generalisation of Lemma~2.5 and Theorem~2.6 in~\cite{bergenthum2007comparison}, where the authors required $f_i$ to be increasing functions in all arguments. %The proof can be found in~\cite{Sootla2015POarxiv}.
\begin{prop}\label{prop:icx-mon-sde} Let the strong solutions to~\eqref{eq:sde} exist. Suppose the functions $f_i(x)$ are convex and such that the functions $f_i(x)$ are increasing in $x_k$ for all $k\ne i$, the functions $\sigma(x)$ are increasing convex functions in the space of semidefinite matrices. Then the process $X(t,X_0)$ propagates the $\cFicx$-order.
\end{prop}
%The authors of~\cite{bergenthum2007comparison} proved their result by first discretising the process with time steps $\Delta$. Then they showed that this discrete time process propagates the $\cFicx$ order \emph{ for all $\Delta>0$}, and they finalised the proof by showing that the discretised process with $\Delta\rightarrow 0$ converges to $X(t,X_0)$. Here we modify the proof of Lemma~2.5 in~\cite{bergenthum2007comparison} by showing that the conditions in Theorem~\ref{prop:icx-mon-sde} are sufficient for the discretised process to propagate the $\cFicx$-order for \emph{a sufficiently small $\Delta$}. Since the proofs are otherwise identical and an adaptation of the proof is presented in the next subsection we refer the reader to~\cite{bergenthum2007comparison,Sootla2015POarxiv}.
\begin{proof}
The authors of~\cite{bergenthum2007comparison} proved their result by first discretising the process with time steps $\Delta$. Then they showed that this discrete time process propagates the $\cFicx$ order \emph{ for all $\Delta>0$}. They have finalised the proof by showing that the discrete process propagates the $\cFicx$ order, and with $\Delta\rightarrow 0$ converges to $X(t,X_0)$. Here we modify the proof of Lemma~2.5 in~\cite{bergenthum2007comparison} by showing that the conditions in Proposition~\ref{prop:icx-mon-sde} are sufficient for the discretised process to propagate the $\cFicx$-order for \emph{a sufficiently small $\Delta$}.

Let $t$ be in $[t_0,T]$. Discretise $[t_0,T]$ into $K$ equidistant points for a sufficiently small distance $\Delta$ between the points. Consider the following discretised process 
\begin{align*}
\tilde X_{K, t_{i+1}} = \tilde X_{K, t_{i}} + f(\tilde X_{K, t_{i}}) \Delta + \sigma(\tilde X_{K, t_{i}}) W_{i}
\end{align*}
where $W_{i}\sim \cN(0,\Delta I)$. Let 
\begin{gather*}
\tilde G_K(t, x) = \Exp\left[ h(\tilde X_{K, t})| \tilde X_{K,t_0} = x\right] \\
G(t, x) = \Exp\left[ h(X_{t})| X_{t_0} = x\right]
\end{gather*}
Due to the existence of strong solutions to~\eqref{eq:sde} we have that $\tilde G_K(t,\cdot)\rightarrow G(t,\cdot)$ with $k\rightarrow\infty$.

We first show the propagation of the $\cFicx$-order by a one-step transition of the discrete process, which is an operator in the following form
% This is essentially showing the propagation of the $\cFicx$-order by the following operator
\begin{equation}
\cT X = X + \Delta f(X) + \sigma(X) W,
\end{equation}
where $W\sim\cN(0,\Delta I)$. Hence, we need to show that for any two random variables $X_1$ and $X_2$ such that $X_1 \leicx X_2$ we have
\begin{align}
\cT X_1 \leicx \cT X_2
\end{align}
Unlike~\cite{bergenthum2007comparison}, where this property is shown for any positive $\Delta$, we will show it only for small increments $\Delta$.  Due to Proposition~\ref{prop:st-o}, we can pick the variables $X_1$ and $X_2$ such that $X_1 \led\Exp(X_2 | X_1)$, without loss of generality. Before showing monotonicity of $\cT$, we need to show some small auxiliary results. 

Let us show that the following relation holds for any realisation $x_1$ of the random variable $X_1$
\begin{multline}\label{ineq:icx-diff}
\Exp\left[X_2 + \Delta f(X_2) + \sigma(X_2) W|x_1\right]  \geicx\\ x_1 + \Delta f(x_1) + \sigma(x_1) W
\end{multline} 
Since the function $f$ is convex so is the function $x + \Delta f(x)$ for any fixed $\Delta$. Hence, by Jensen's inequality we have $\Exp(X_2 + \Delta f(X_2)| x_1) \succeq_x \Exp(X_2|x_1) + \Delta f(\Exp(X_2)| x_1)$. Due to Lipschitz continuity of $f(x)$, we can find a small enough $\Delta$ such that $\Delta(f(x)-f(y)) \le x-y$ for $x\succeq_x y$. Hence the function $x + \Delta f(x)$ is also increasing in $x$ for a sufficiently small $\Delta$. Finally, we have that for a small $\Delta$
\begin{gather*}
x_1 + \Delta f(x_1)\preceq_x \Exp(X_2|x_1) + \Delta f(\Exp(X_2)| x_1) \\
\preceq_x \Exp(X_2 + \Delta f(X_2)| x_1)
\end{gather*}
since $x_1 \preceq_x \Exp(X_2| x_1)$ for every realisation $x_1$ of the random vector $X_1$. By Jensen's inequality and monotonicity of $\sigma$ we have $\Exp(\sigma(X_2)|x_1) \gepsd \sigma(\Exp(X_2|x_1)) \gepsd\sigma(x_1)$. Moreover, $\Exp(\sigma(X_2)|x_1)\Exp(\sigma(X_2)|x_1)^T \gepsd \sigma(x_1)\sigma(x_1)^T$. Therefore, by Proposition~\ref{prop:st-o} we have $\sigma(x_1) W \leicx \Exp(\sigma(X_2)|x_1) W$ and consequently the comparison~\eqref{ineq:icx-diff} holds. 

Now consider the following chain of inequalities, where Jensen's inequality is used again 
\begin{gather*}
\Exp h(\cT X_2) = \Exp\,\Exp \left[h(\cT X_2 )| X_1\right] \ge \Exp \left[ h(\Exp \left[\cT X_2 | X_1 \right]) \right] = \\
\Exp h\left(\Exp \left[X_2 + \Delta f(X_2) + \sigma(X_2) W|X_1 \right] \right) = \\
\Exp\int h\left( \Exp\left[X_2 + \Delta f(X_2) + \sigma(X_2) W|x_1\right] \right) P(dx_1)\ge \\
\Exp\int h\left( x_1 + \Delta f(x_1) + \sigma(x_1) W \right) P(dx_1) =  \Exp h(\cT X_1).
\end{gather*}
where the last inequality holds due to~\eqref{ineq:icx-diff} by definition of the $\cFicx$-order propagation. Hence we have shown that $\cT X_2\geicx \cT X_1$, if $X_2 \geicx X_1$. Now it is left to show that $\tilde \cG_K(t_0,x)$ belongs to $\cFicx$ for sufficiently small increments $\Delta$. By the Markovian property of $\tilde \cG_K(t_0, \cdot)$ we have 
\begin{multline*}
\tilde\cG_K(t, y) = \Exp(h(\tilde X_{K,t})| \tilde X_{K,t_0} = y) = \\
\Exp h(\cT_{t-\Delta}\dots \cT_{t_0} y)
\end{multline*}
Let $Y$ be a Bernoulli distribution for $y_1$, $y_2$ with a distribution $P_Y(Y = y_1) = \alpha$,  $P_Y(Y= y_2) = (1-\alpha)$, and the mean $\Exp Y = \alpha y_1 + (1 -\alpha) y_2$. Clearly $\Exp Y \leicx Y$. Using the propagation of the $\cFicx$-order by the operator $\cT$ for all $t\in[t_0,T]$ and a small $\Delta$ we get:
\begin{multline*}
\tilde \cG_K(t, \alpha y_1 + (1-\alpha) y_2) = \tilde \cG_K(t, \Exp Y) = \\
\Exp h(\cT_{t-\Delta} \dots \cT_{t_0}\Exp Y) \le  \Exp h(\cT_{t-\Delta} \dots \cT_{t_0} Y) = \\
\Exp \tilde \cG_K(t, Y) =  \alpha \tilde \cG_K(t, y_1 ) + (1-\alpha)\tilde \cG_K(t, y_2 ) 
\end{multline*}
Now for every $x\le y$ by the propagation of the $\cFicx$-order by $\cT$ we get:
\begin{multline*}
\tilde \cG_K(t, x) = \Exp h(\cT_{t-\Delta} \dots \cT_{t_0} x) \le \\
\Exp h(\cT_{t-\Delta} \dots \cT_{t_0} y) = 
\tilde \cG_K(t, y)
\end{multline*}
Hence, $\tilde \cG_K(t,\cdot)$ belongs to the class $\cFicx$, which implies that $\cG(t,\cdot)$ belongs to $\cFicx$ as well and completes the proof.
\end{proof}

\subsection{Remarks on Biochemical Processes Propagating Orders}
Biochemical networks are typically modelled by a continuous time infinite Markov chain, probability distribution function of which is computed by a Chemical Master Equation (CME):
\[ 
 \frac{\partial \Pr(n, t)}{\partial t} = \sum_{i=1}^R  v_i (n - S_{i}) \Pr(n - S_i, t)- v_i(n) \Pr(n, t),
\]
where $R$ is the number of reactions; %$N$ - number of species; 
column vectors $S_{i}$ form a stoichiometry matrix $S$; $v_i$ are the reaction rates; $n$ is a vector containing the number of molecules $n_j$ of species $j$; finally, $\Pr(n,t)$ is the probability of the vector of the number of molecules equal to $n$ at time $t$. This equation cannot be solved analytically except for a handful of cases and numerical simulations are extremely expensive. In order to lower the complexity of simulations, different approximations of a CME are often employed. One of such approximations is a Chemical Langevin Equation (CLE)~\cite{gillespie2000langevin}:
\begin{equation}
  \label{eq:cle}
   d X = S v(X)dt + S V(X) d W, 
\end{equation}
where $d W$ is a vector of Brownian motions, and $V(X) = {\diag}(\sqrt{v(X)}))$. In this case, one approximates a Poisson distribution with a Gaussian, which is valid when the average number of firing reactions per unit time is large. One of the technical problems with the CLE is that $v(X)$ can take negative values with a small, but non-zero probability. In order to avoid these technical problems, we illustrate the propagation of order on a Linear Noise Approximation, which is valid for reactions in a large volume~\cite{van2007stochastic}.
\begin{gather}
  \label{eq:lna-m}   d x = S v(x)dt, x(0) = x_0\\
  \label{eq:lna-n}   d \eta = J(x) \eta dt + S V(x) dW,
\end{gather}
where $W$ is a Brownian motion, $x(t)$ is a deterministic variable, $V_{i i}(x) = \sqrt{v_i(x)}$, $V_{i j}(x) = 0$ for all $i\ne j$, and $J(x)$ is a Jacobian of $S v(x)$. 

Consider a set of unimolecular reactions, that is 
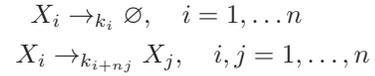
\begin{gather*}
X_i \rightarrow_{k_i} \varnothing, \quad i=1,\dots n \\
\qquad X_i \rightarrow_{k_{i + n j}} X_j,\quad i,j=1,\dots, n 
\end{gather*}
with the vector of reaction rates in the following form 
\begin{gather*}
v_i(x) = k_{i} x_{i}, \quad i = 1,\dots,n \\
v_{i + n j} = k_{i + n j} x_i,\quad i,j=1,\dots,n, i\ne j
\end{gather*}
It is straightforward to verify that we will have $S v(x) = A x$, where 
\begin{gather*}
A_{i j} =\begin{cases}
k_{i + n j}, & i\ne j,\\
-k_i -\sum_{j=1,j\ne i} k_{j + n i}& i=j.
\end{cases}
\end{gather*}
and $S V(x) V(x) S^T$ is linear in $x$. consider the elements of the diagonal of $S V(x)^2 S^T$, which are equal to $c_{i i} =\sum\limits_{j} S_{i j}^2 v_j(x)$. Note that by Proposition~\ref{prop:fd-prop} the function $c_{i i}$ can depend only on $x_i$, which implies that if there exists $v_j$ which depends on $x_m$ for $m\ne i$, then $S_{i j} = 0$. Hence the vector $S v(x)$ can be rewritten as $w(x)$, where $w_i(x)$ depends only on $x_i$, which is a birth-death process. It is straightforward to verify that in general only birth-death processes propagate the $\cFd$-order, which limits the scope of potential applications. At the same time, we have the following result for the propagation of the $\cFicx$-order.
\begin{prop}\label{prop:lna-prop}
A linear noise approximation of any network of unimolecular reactions is propagating the $\cFicx$-order. 
\end{prop}
\begin{proof} 
This is an adaptation of the proof of Lemma~2.5 and Theorem~2.6 in~\cite{bergenthum2007comparison}, hence we will only provide a sketch of the proof.  Discretise $[t_0,T]$ into $K$ equidistant points with a distance $\Delta$ between the points. Consider the following discretised process with $t_i\in[t_0, T]$
\begin{align}
\tilde x_{t_{i+1}} &= \tilde x_{t_{i}} + \Delta A \tilde x_{t_{i}} \\
\tilde \eta_{t_{i+1}} &= \tilde \eta_{t_{i}} + \Delta A \tilde \eta_{t_{i}}  + S V(\tilde x_{t_{i}}) W
\end{align}
where $W\sim \cN(0,\Delta I)$. Let 
{\small \begin{gather*}
Z_{t_{i}} = \begin{pmatrix}
\tilde x_{t_{i}} \\ \tilde \eta_{t_{i}}
\end{pmatrix} \tilde A = \begin{pmatrix}
A  &0\\
0 & A 
\end{pmatrix} \tilde B(Z_{t_i}, W) = \begin{pmatrix}
0 \\
S V(\tilde x_{t_{i}}) W
\end{pmatrix} 
\end{gather*}}

First we show the propagation of the $\cFicx$-order by a one-step transition operator $\cT$ of the discrete process, which is as follows
\begin{gather*}
\cT Z = Z + \Delta \tilde A Z + \tilde B(Z, W),
\end{gather*}
Let $Z_{1} = \begin{pmatrix}
x^T & \eta^T
\end{pmatrix}^T$, $Z_{2} = \begin{pmatrix}
y^T & \xi^T
\end{pmatrix}^T$, we need to show that $\cT Z_1 \leicx \cT Z_2$ follows from $Z_1 \leicx Z_2$. Without loss of generality, by Proposition~\ref{prop:st-o} we can pick the variables $Z_1$ and $Z_2$ such that $Z_1 \led\Exp(Z_2 | Z_1)$. Now for a small enough $\Delta$ the matrix $I +\Delta \tilde A$ is nonnegative, hence 
\begin{gather}\label{eq:drift-icx-pf}
\Exp\left[Z_2 + \Delta \tilde A Z_2 | z_1\right] \succeq_x z_1 + \Delta \tilde A z_1.
\end{gather}
for any realisation $z_1$ of the random variable $Z_1$. Consider now the terms $\tilde B(Z_2, W)$ and $\tilde B(z_1, W)$. Note that $v(x)\succeq_x v(y)$ for all $x \succeq_x y$, since $V^2$ is a linear function of $x$, which takes only positive values for positive $x$. This implies that $\xi^T V^2(x) \xi \ge \xi^T V^2(y) \xi$ for any vector $\xi$ and therefore $S V^2(x) S^T \gepsd S V^2(y) S^T$ for all $x \succeq_x y$. Hence by Proposition~\ref{prop:st-o}, we have that $\Exp\left[\tilde B(Z_2,W)| z_1\right] \geicx \tilde B(z_1,W)$, which together with~\eqref{eq:drift-icx-pf} results in
\begin{multline*}
 \Exp\left[Z_2 + \Delta \tilde A Z_2 + \tilde B(Z_2, W) |z_1\right] \geicx\\ z_1 + \Delta \tilde A z_1 + \tilde B(z_1, W)
\end{multline*} 
for any realisation $z_1$ of the random variable $Z_1$. Now by Jensen's inequality and the inequality above, we have 
{\small \begin{gather*}
\Exp h(\cT Z_2) = \Exp\,\Exp \left[h(\cT Z_2 )| Z_1\right] \ge \Exp \left[ h(\Exp \left[\cT Z_2 | Z_1 \right] \right] = \\
\Exp h\left(\Exp \left[Z_2 + \Delta \tilde A Z_2 + \tilde B(Z_2, W) |Z_1 \right] \right) = \\
\Exp \int h\left( \Exp\left[Z_2 + \Delta \tilde A Z_2 + \tilde B(Z_2, W) | z_1\right] \right) P(d z_1)\ge \\
\Exp \int h\left( z_1 + \Delta \tilde A z_1 + \tilde B(z_1, W) \right) P(d z_1) = \Exp h(\cT Z_1),
\end{gather*}}
where $z_1$ is a realisation of the random variable $Z_1$. It remains to show that $\Exp[h(Z_{t_i})| Z_{t_0} = z_0]$ belongs to the $\cFicx$ class for any $t_i$. This part of the proof is equivalent to the proof of Proposition~\ref{prop:icx-mon-sde} and the proofs in~\cite{bergenthum2007comparison} and is therefore omitted.
\end{proof}

This proposition indicates that the $\cFicx$ order is propagated by the mean and the covariance matrix of the process. Hence, by comparing initial distributions, we can compare distributions of the process at any point of time. This result is naturally related to Proposition~\ref{prop:st-o}, where it is shown that the $\cFd$-order is equivalent to the almost surely comparison, while the $\cFicx$-order is equivalent to the comparison of conditional expectations. At the same the $\cFicx$-order is still quite restrictive and the class of networks propagating the $\cFicx$-order can still be too small. However, Proposition~\ref{prop:lna-prop} shows that imposing additional constraints on the propagation class $\cF$ in comparison with $\cFd$ entails a larger class of biochemical network with a certain monotonicity property. It remains to establish, what kind of a class of functions to consider, in order to induce a valuable order.

\section{Conclusion and Discussion}
In this paper a step is taken towards a better understanding of monotonicity (propagation of order) of Markov processes. Some novel results for the deterministic processes are derived, however, the main application in mind is \emph{stochastic} biological processes. In this paper, we make a step forward towards understanding the biological processes propagating the order. We do so by considering Linear Noise Approximation (of a Master Equation), which is a Gaussian process. In this paper, it is shown that any LNA propagating the $\cFd$-order is a collection of decoupled birth-death processes. At the same time any LNA process describing unimolecular reactions is propagating the $\cFicx$-order, which entails that the order is propagated by the mean and the covariance matrix of the process. 
\bibliography{Biblio}
\end{document}